\documentclass{article}[11pt,a4paper]

\usepackage[frenchb]{babel}
\usepackage{amsmath}
\usepackage{amsfonts}
\usepackage{amsthm}
\usepackage{amssymb}
\usepackage{bm}
\usepackage[active]{srcltx}

\usepackage[latin1]{inputenc}

\usepackage{enumerate}
\usepackage{multicol}

\usepackage[all,cmtip]{xy}

\newtheoremstyle{mystyle2}{}{}{}{2pt}{\scshape}{.}{ }{}
\newtheoremstyle{mystyle}{}{}{\slshape}{2pt}{\scshape}{.}{ }{}
\newtheoremstyle{etapestyle}{}{}{\itshape}{2em}{\sffamily}{:}{ }{\thmname{#1}}
\newtheoremstyle{definitionstyle}{}{}{}{2pt}{\bfseries}{.}{ }{}

\newtheorem{thm}{Th\'{e}or\`{e}me}[section]
\newtheorem{cor}[thm]{Corollaire}

\newtheorem{prop}[thm]{Proposition}

\newtheorem{lemme}[thm]{Lemme}

\newtheorem*{prop*}{Proposition}

\theoremstyle{mystyle2}
\newtheorem{ex}[thm]{Exemple}

\theoremstyle{mystyle}

\theoremstyle{remark}
\newtheorem{rem}[thm]{Remarque}

\theoremstyle{etapestyle}

\theoremstyle{definitionstyle}

\DeclareMathOperator{\codim}{codim} 
 \DeclareMathOperator{\Spec}{Spec}
 \DeclareMathOperator{\lon}{long}

\DeclareMathOperator{\Pic}{Pic}
\DeclareMathOperator{\Card}{Card}

\DeclareMathOperator{\Sym}{Sym}
\DeclareMathOperator{\Id}{Id}

\DeclareMathOperator{\rad}{rad}

 \DeclareMathOperator{\Sing}{Sing}

\begin{document}

\renewcommand{\bibname}{Bibliographie}

\title{Degr\'es d'homog\'en\'eit\'e de l'ensemble des intersections compl\`etes singuli\`eres}
\author{Olivier BENOIST}  
\date{}
\maketitle


\tableofcontents

\newpage

\section{Introduction}

On travaille sur un corps de base $K$, qui sera souvent sous-entendu. Par exemple, $\mathbb{P}^N=\mathbb{P}^N_K$.

Une formule classique de Boole montre que, si $K$ est de caract\'eristique $0$, l'ensemble des 
hypersurfaces singuli\`eres de degr\'e $d$ dans $\mathbb{P}^N$ est un diviseur de degr\'e $(N+1)(d-1)^N$ 
dans l'espace projectif de toutes les hypersurfaces. 
On obtient ici des formules analogues pour des intersections compl\`etes 
de codimension et de degr\'es quelconques dans $\mathbb{P}^N$, en toute caract\'eristique.

\subsection{\'Enonc\'e du th\'eor\`eme principal}\label{situation}

On fixe $1\leq c\leq N+1$ et $1\leq d_1,\ldots,d_c$ des entiers. On notera $e_i=d_i-1$. On va s'int\'eresser aux intersections compl\`etes de codimension $c$ dans 
$\mathbb{P}^N$, solutions d'\'equations homog\`enes de degr\'es $d_1,\ldots,d_c$ : on notera $n=N-c$ leur dimension.

Pour cela, on consid\`ere $V=\bigoplus_{1\leq i\leq c}H^0(\mathbb{P}^N,\mathcal{O}(d_i))$. 
Les \'el\'ements de $V$ sont la donn\'ee de $c$ polyn\^omes homog\`enes de degr\'es $d_1,\ldots,d_c$ en $N+1$ variables $X_0,\ldots,X_N$. 
Soit $D$ le ferm\'e de $V$ constitu\'e des $(F_1,\ldots,F_c)$ tels que $\{F_1=\ldots=F_c=0\}$ ne soit pas lisse de codimension $c$ dans $\mathbb{P}^N$. 
On le munit de sa structure r\'eduite.
Par le lemme \ref{irred}, la vari\'et\'e $D$ est irr\'eductible. Notons $\Delta$ une de ses \'equations (par convention, $\Delta=1$ si $\codim_V(D)>1$).
On appellera $D$ le lieu discriminant et $\Delta$ le discriminant.

\begin{rem}\label{remdef}
 Les lemmes \ref{nondef} et \ref{def}, au vu du corollaire \ref{cridef}, montreront que $\codim_V(D)>1$ exactement quand $d_1=\ldots=d_c=1$ et $c<N+1$.
\end{rem}

Le discriminant est visiblement homog\`ene en les coefficients de chacune des \'equations $F_i$. Le but de ce texte est 
de calculer ces degr\'es d'homog\'en\'eit\'e partiels.
Soyons plus pr\'ecis.

Plusieurs transformations de $V$ laissent $D$ invariant. C'est le cas des actions $\rho$ et $\rho_i$ de $\mathbb{G}_m$ et $\rho'$ de $GL_{N+1}$ d\'ecrites ci-dessous :
\begin{alignat}{3}
   \rho(\lambda):(F_1,\ldots,F_c)&\mapsto(\lambda F_1,\ldots,\lambda F_c) \nonumber\\
   \rho_i(\lambda):(F_1,\ldots,F_c)&\mapsto(F_1,\ldots,F_{i-1},
\lambda F_{i},F_{i+1},\ldots,F_c) \label{action}\\
   \rho'(M):(F_1,\ldots,F_c)&\mapsto(F_1\circ M^{-1},\ldots,F_c\circ M^{-1}).\nonumber
\end{alignat}
Les actions duales induites sur $\Sym^{\bullet}V^*$ pr\'eservent la droite $\left\langle \Delta\right\rangle$. Ces actions sur $\left\langle \Delta\right\rangle$ 
se font via un caract\`ere du groupe. Il existe donc des entiers $\deg$, $\deg_i$ et $\deg_{var}$ tels que 
\begin{alignat}{3}
   \rho(\lambda).\Delta&=\lambda^{-\deg}\Delta \nonumber\\
   \rho_i(\lambda).\Delta&=\lambda^{-\deg_i}\Delta \label{defdeg}\\
   \rho'(M).\Delta&=\det(M)^{\deg_{var}}\Delta.\nonumber
\end{alignat}
Par exemple, $\deg$ est le degr\'e total du polyn\^ome homog\`ene $\Delta$. Les autres nombres s'interpr\`etent comme des degr\'es d'homog\'en\'eit\'e partiels. 
Les identit\'es  
$\rho(\lambda)=\rho_1(\lambda)\circ\ldots\circ\rho_c(\lambda)$ et $\rho'(\lambda^{-1}\Id)=\rho_1(\lambda)^{\circ d_1}\circ\ldots\circ\rho_c(\lambda)^{\circ d_c}$
montrent qu'ils sont li\'es par les relations :
\begin{alignat}{2}
   \deg &=\sum_{i=1}^{c} \deg_i \label{deg}\\
   (N+1)\deg_{var}&=\sum_{i=1}^{c} d_i\deg_i.\label{degvar}
\end{alignat}

\begin{rem}
On peut aussi d\'ecrire les $\deg_i$ comme suit. Consid\'erons comme espace de param\`etres 
$E=\prod_{1\leq i\leq c}\mathbb{P}H^0(\mathbb{P}^N,\mathcal{O}(d_i))$ et notons encore $D$ le ferm\'e irr\'eductible de $E$
constitu\'e des $([F_1],\ldots,[F_c])$ tels que $\{F_1=\ldots=F_c=0\}$ ne soit pas lisse de codimension $c$ dans $\mathbb{P}^N$, muni de sa structure r\'eduite.
Alors, quand $D$ est un diviseur de $E$, sa classe dans $\Pic(E)\simeq\mathbb{Z}^c$ est $\mathcal{O}(\deg_1,\ldots,\deg_c)$.
\end{rem}

Le r\'esultat principal de ce texte est le suivant :

\begin{thm}\label{princ} On a les \'egalit\'es suivantes :
\begin{alignat}{2}
   \deg_i =\frac{1}{\mu}d_1\ldots\hat{d}_i\ldots d_c&\sum_{l=1}^{c}\frac{1}{\prod_
{l'\neq l}(e_l-e_{l'})}\left(\frac{e_i^{N+1}-e_l^{N+1}}{e_i-e_l}\right)\label{degi} \\
   \deg_{var}=\frac{1}{\mu}d_1\ldots d_c&\sum_{l=1}^{c}\frac{e_l^N}{\prod_{l'\neq l}(e_l-e_{l'})}\label{var}
\end{alignat}
o\`u $\mu=1$ si $K$ n'est pas de caract\'eristique $2$ ou si $n$ est impair, et $\mu=2$ si $K$ est de caract\'eristique $2$ et $n$ est pair.
\end{thm}

\begin{rem}
Il faut interpr\'eter cet \'enonc\'e, comme tous les \'enonc\'es similaires de ce texte, de la mani\`ere suivante : 
le terme de gauche est une fonction polynomiale en les $d_l$ dont le polyn\^ome est donn\'e par le terme de droite.
Que le terme de droite soit un polyn\^ome en les $d_l$ est cons\'equence du lemme \ref{polyn} $(i)$.
\end{rem}

\begin{rem}\label{remconv}
Dans l'\'egalit\'e (\ref{degi}), il faut interpr\'eter $\frac{e_i^{N+1}-e_i^{N+1}}{e_i-e_i}$ comme une notation pour le polyn\^ome $(N+1)e_i^N$.

\end{rem}

\begin{rem}

Par (\ref{deg}), on peut \'egalement calculer $\deg$. 
Cependant, on ne peut simplifier avantageusement l'expression obtenue. M\^eme pour $c=2$, 
l'exemple \ref{c2} montre que la formule pour $\deg$ est nettement moins \'el\'egante que celles pour $\deg_{var}$, $\deg_1$ et $\deg_2$.

\end{rem}

On d\'eduira de ce th\'eor\`eme le r\'esultat qui suit :

\begin{thm}\label{princ2}
Soit $K=\mathbb{Q}$, de sorte qu'on peut consid\'erer le discriminant $\Delta$ comme un polyn\^ome irr\'eductible \`a coefficients entiers. Soit $p$ un nombre premier. 
Alors la r\'eduction modulo $p$ de $\Delta$ est irr\'eductible si $p\neq 2$ ou si $n$ est impair. C'est le carr\'e d'un polyn\^ome irr\'eductible 
si $p=2$ et $n$ est pair.
\begin{item}
 
\end{item}

\end{thm}

\subsection{Quelques exemples}

On commence par illustrer le th\'eor\`eme \ref{princ}. 
On suppose pour simplifier que $\mu=1$. Par exemple, on peut prendre $K$ de caract\'eristique diff\'erente de $2$.

\begin{ex}

Quand $c=1$, $D$ est l'ensemble des \'equations d'hypersurfaces singuli\`eres de degr\'e $d_1$ dans $\mathbb{P}^N$ : 
$\Delta$ est donc le discriminant usuel. Les formules (\ref{degi}) et (\ref{var}) montrent qu'on a alors :
\begin{alignat*}{2}
   \deg_1 &=(N+1)e_1^N\\
   \deg_{var}&=e_1^N.
\end{alignat*}
On retrouve la formule classique de Boole, qu'on pourra par exemple trouver dans \cite{GKZ} Chap.1, 4.15 et Chap.9, 2.10a.

\end{ex}

\begin{ex}

Quand $c=N+1$, $D$ est l'ensemble des $N+1$-uplets d'\'equations homog\`enes admettant une solution commune dans $\mathbb{P}^N$ : 
$\Delta$ est donc le r\'esultant usuel. 

Pour \'evaluer la formule (\ref{degi}) dans ce cas, on peut remarquer que la quantit\'e $\sum_{l=1}^{c}\frac{1}{\prod_
{l'\neq l}(e_l-e_{l'})}\left(\frac{e_i^{N+1}-e_l^{N+1}}{e_i-e_l}\right)$ est un polyn\^ome de degr\'e $0$ en les $e_j$, 
c'est-\`a-dire une constante. Pour la calculer, on peut faire par exemple $d_j=j$, de sorte que $\prod_
{l'\neq l}(e_l-e_{l'})=(-1)^{N+1-l}(l-1)!(N+1-l)!$ . L'expression obtenue s'\'evalue facilement \`a l'aide de la formule 
du bin\^ome. Tous calculs faits, cette constante vaut $1$, et on obtient :
\begin{alignat*}{2}
   \deg_i &=d_1\ldots\hat{d}_i\ldots d_{N+1}\\
   \deg_{var}&=d_1\ldots d_{N+1}.
\end{alignat*}
L\`a encore, ces formules sont classiques. On peut les trouver dans \cite{GKZ} Chap.13, 1.1.

\end{ex}

\begin{ex}

Quand $d_1=\ldots=d_c=d$, $D$ est l'ensemble des syst\`emes lin\'eaires de degr\'e $d$ et de dimension $c-1$ dont le lieu de base est singulier. 

Dans ce cas particulier \'egalement, les formules g\'en\'erales se simplifient nettement. Comme, par sym\'etrie, tous les $\deg_i$ sont \'egaux, 
il suffit par (\ref{degvar}) et (\ref{deg}) de calculer $\deg_{var}$. Dans l'expression (\ref{var}), $\sum_{l=1}^{c}\frac{e_l^N}{\prod_{l'\neq l}(e_l-e_{l'})}$ 
est un polyn\^ome homog\`ene de degr\'e $N-c+1$ en les $e_j$. Par cons\'equent, $\deg_{var}=\lambda d^c e^{N-c+1}$ o\`u $\lambda$ est une constante \`a calculer. 
On \'evalue cette constante en faisant $d_l=l\varepsilon$ dans le polyn\^ome $\sum_{l=1}^{c}\frac{e_l^N}{\prod_{l'\neq l}(e_l-e_{l'})}$, en calculant 
cette quantit\'e \`a l'aide de la formule du bin\^ome, puis en faisant tendre $\varepsilon$ vers $0$. Tous calculs faits, 
\begin{alignat*}{2}
   \deg_i &=\binom{N+1}{c}d^{c-1}(d-1)^{N-c+1}\\
   \deg_{var}&=\binom{N}{c-1}d^{c}(d-1)^{N-c+1}.
\end{alignat*}
Ces formules auraient aussi pu \^etre obtenues \`a l'aide de r\'esultats de \cite{GKZ}, par exemple de Chap.13, 2.5.

Supposons de plus $d=1$ et $c<N+1$. Les degr\'es s'annulent : cela correspond aux cas o\`u $D$ est de codimension $>1$ dans $V$ (voir la remarque \ref{remdef}).

Supposons enfin $d=1$ et $c=N+1$, le polyn\^ome $\Delta$ est le d\'eterminant usuel d'une matrice de taille $N+1$ ; 
on retrouve son degr\'e total $\deg=N+1$.
\end{ex}

\begin{ex}\label{c2}

Sp\'ecialisons maintenant les formules (\ref{degi}) et (\ref{var}) au cas d'intersections compl\`etes de codimension $2$, c'est-\`a-dire $c=2$. Il vient :
\begin{alignat*}{3}
   \deg_1 &=d_2(e_2^{N-1}+2e_1e_2^{N-2}+\ldots+Ne_1^{N-1})\\
\deg_2 &=d_1(e_1^{N-1}+2e_2e_1^{N-2}+\ldots+Ne_2^{N-1})\\
   \deg_{var}&=d_1d_2\frac{e_2^N-e_1^N}{e_2-e_1}.
\end{alignat*}
\end{ex}

Illustrons enfin le th\'eor\`eme \ref{princ2} en explicitant un cas particulier classique :

\begin{ex}

Quand $N=1$ et $c=1$, $\Delta$ est le discriminant usuel d'un polyn\^ome en une variable, vu comme un polyn\^ome homog\`ene en deux variables. 
Sa description classique en fonction des racines de ce polyn\^ome 
(voir par exemple \cite{Prasolov} 1.3.2) montre que ce polyn\^ome est irr\'eductible en caract\'eristique diff\'erente de $2$,
et le carr\'e d'un polyn\^ome irr\'eductible en caract\'eristique $2$. Comme $n=0$ est pair, c'est ce que pr\'edit le th\'eor\`eme \ref{princ2}.

Signalons le cas particulier bien connu o\`u $d=2$. Le discriminant du polyn\^ome $aX^2+bX+c$ est $b^2-4ac$. 
C'est toujours irr\'eductible, sauf en caract\'eristique $2$, $b^2$ \'etant visiblement un carr\'e.
\end{ex}

\subsection{Strat\'egie de la d\'emonstration}

Dans la preuve du th\'eor\`eme \ref{princ}, on peut, quitte \`a le remplacer par une cl\^oture alg\'ebrique, choisir $K$ alg\'ebriquement clos.
On commence de plus par supposer $K$ de caract\'eristique $0$.

La proposition \ref{descridual} permet d'interpr\'eter 
$D$ comme 
la vari\'et\'e duale d'une vari\'et\'e torique lisse convenable.
 Dans leur livre \cite{GKZ}, Gelfand, Kapranov et Zelevinsky ont \'etudi\'e
 ces vari\'et\'es ; ils obtiennent notamment une formule combinatoire 
permettant de calculer le degr\'e de la vari\'et\'e duale d'une vari\'et\'e torique 
lisse (\cite{GKZ} Chap.9, 2.8).

  Dans la deuxi\`eme partie de ce texte, on g\'en\'eralise cet \'enonc\'e pour obtenir une
formule analogue calculant des degr\'es d'homog\'en\'eit\'e partiels : c'est l'objet du 
th\'eor\`eme \ref{char}. On utilise de mani\`ere cruciale les r\'esultats de \cite{GKZ}.

Dans la troisi\`eme partie de ce texte, on d\'emontre le th\'eor\`eme \ref{princ} en caract\'e\-ristique $0$ en \'evaluant cette formule dans notre cas particulier. 
Les calculs sont men\'es dans les propositions \ref{premprinc} et \ref{deuxprinc}.
\\

Finalement, on explique dans la quatri\`eme partie les modifications \`a apporter \`a la preuve du th\'eor\`eme \ref{princ} pour qu'elle 
fonctionne en toute caract\'eristique. 

Le seul obstacle que l'on rencontre sont les sp\'ecificit\'es de la th\'eorie de la dualit\'e projective en caract\'eristique finie, 
qui ont \'et\'e \'etudi\'ees tout d'abord par Wallace \cite{Wallace}, et 
au sujet desquelles on pourra consulter le survey \cite{Kleiman} de Kleiman. Elles ont pour cons\'equence que le r\'esultat de \cite{GKZ} que nous utilisons 
a besoin d'\^etre l\'eg\`erement modifi\'e pour valoir en caract\'eristique finie. Cette modification effectu\'ee, la preuve est identique.

Il convient de remarquer que si ces arguments suppl\'ementaires sont indispensables en toute caract\'eristique finie, le th\'eor\`eme \ref{princ} ne
voit son \'enonc\'e  modifi\'e qu'en caract\'eristique $2$.

Enfin, le th\'eor\`eme \ref{princ2} se d\'eduit ais\'ement du th\'eor\`eme \ref{princ}.

\section{Duale d'une vari\'et\'e torique}\label{part1}

Dans cette partie, $K$ est suppos\'e alg\'ebriquement clos de caract\'eristique $0$.

L'objectif est de montrer le th\'eor\`eme \ref{char}. 
\'Etant donn\'ee une vari\'et\'e torique projective lisse, 
on calcule combinatoirement l'action du tore sur l'\'equation de sa vari\'et\'e duale.

\subsection{Notations}

On commence par fixer des notations.

Soit $0\to\mathbb{G}_m\to\tilde{T}\to T\to 0$ une suite exacte courte de tores, $\tilde{T}$ \'etant de dimension $k\geq1$ et $T$ 
de dimension $k-1$. Soit $0\to \mathfrak{X}(T)\to \mathfrak{X}(\tilde{T})\stackrel{h}{\rightarrow}\mathbb{Z}\to 0$ la suite exacte courte de leurs 
groupes de caract\`eres. On notera $\mathfrak{X}_1=h^{-1}(1)$ : c'est un espace principal homog\`ene sous $\mathfrak{X}(T)$.

Soit $A=\{\chi_1,\ldots,\chi_{|A|}\}\subset \mathfrak{X}_1$ un sous-ensemble fini engendrant $\mathfrak{X}_1$ comme espace affine. 
On notera $\tilde{X}_A\subset(K^A)^*$ la vari\'et\'e torique affine de tore $\tilde{T}$ associ\'ee. 
C'est, par d\'efinition, l'adh\'erence des $(\chi_1(t),\ldots,\chi_{|A|}(t))$, $t\in\tilde{T}$. 
La vari\'et\'e $\tilde{X}_A$ est un c\^one ; on peut consid\'erer son projectivis\'e $X_A\subset\mathbb{P}((K^A)^*)$ 
qui est une vari\'et\'e torique projective de tore $T$. Le polytope correspondant est l'enveloppe convexe 
$Q$ de $A$ dans $\mathfrak{X}_{1,\mathbb{R}}=\mathfrak{X}_1\otimes_{\mathbb{Z}}\mathbb{R}$.

Notons $X^{\vee}_A\subset\mathbb{P}(K^A)$ la vari\'et\'e duale de $X_A$ munie
de sa structure r\'eduite et $\tilde{X}^{\vee}_A\subset K^A$ son c\^one affine. On note $\Delta_A$ 
une \'equation homog\`ene de $X^{\vee}_A$ (par convention, $\Delta_A=1$ si $X_A$ est d\'efective, c'est-\`a-dire si $\codim_{\mathbb{P}(K^A)}(X^{\vee}_A)>1$).
Le tore $\tilde{T}$ agit sur $(K^A)^*$ en pr\'eservant $\tilde{X}_A$. Par cons\'equent, 
l'action duale $t\cdot \sum c_a\chi_a\mapsto\sum c_a\chi_a^{-1}(t)\chi_a$ de $\tilde{T}$ 
sur $K^A$ pr\'eserve $\tilde{X}^{\vee}_A$. L'action induite de $\tilde{T}$ sur $\Sym^{\bullet}(K^A)^*$ 
pr\'eserve donc la droite $\left\langle\Delta_A \right\rangle$ ; cette action se fait via un caract\`ere 
de $\tilde{T}$ que l'on notera $\Xi_A$ : $t\cdot\Delta_A=\Xi_A(t)\Delta_A$. Ainsi, on a :
\begin{equation}\label{defxi}
\Delta_A(\sum c_a\chi_a(t)\chi_a)=\Xi_A(t)\Delta_A(\sum c_a\chi_a).
\end{equation}
En sp\'ecialisant cette identit\'e \`a des sous-groupes \`a un param\`etre convenables, on peut calculer les degr\'es d'homog\'en\'eit\'e 
partiels de $\Delta_A$. Par exemple, en sp\'ecialisant au sous-groupe \`a un param\`etre correspondant au cocaract\`ere $h$, il vient 
$\Delta_A(\lambda f)=\lambda^{h(\Xi_A)}\Delta_A(f)$ pour tout $f\in K^A$, soit $$\deg(\Delta_A)=h(\Xi_A).$$

\subsection{\'Equation de la duale}

  Le th\'eor\`eme \cite{GKZ} Chap. 9, 2.8  d\^u \`a Gelfand, Kapranov et Zelevinsky
permet de calculer $\deg(\Delta_A)=h(\Xi_A)$ 
en fonction du polytope $Q$ si $X_A$ est lisse.
Nous allons g\'en\'eraliser cet \'enonc\'e en obtenant une formule pour $\Xi_A$. La
d\'emonstration est tr\`es proche de celle de \cite{GKZ}, et utilise de mani\`ere essentielle les r\'esultats de ce livre. 
En particulier, elle repose sur une formule pour $\Delta_A$ que nous rappelons dans ce paragraphe.

Si $\Gamma$ est une face de $Q$, on consid\`erera $\tilde{\Gamma}$ le c\^one sur $\Gamma$ de sommet $0$ dans 
$\mathfrak{X}(\tilde{T})_{\mathbb{R}}$, $\Gamma^0$ et $\tilde{\Gamma}^0$ leurs int\'erieurs relatifs, $\Gamma_{\mathbb{R}}$ 
le sous-espace affine de $\mathfrak{X}_{1,\mathbb{R}}$ engendr\'e par $\Gamma$, 
et $\tilde{\Gamma}_{\mathbb{R}}$ le sous-espace vectoriel 
de $\mathfrak{X}(\tilde{T})_{\mathbb{R}}$ engendr\'e par $\tilde{\Gamma}$.
L'espace affine $\Gamma_{\mathbb{R}}$ est muni du r\'eseau naturel $\Gamma_{\mathbb{Z}}=\Gamma_{\mathbb{R}}\cap \mathfrak{X}_1(T)$ ; on notera $\mu_{\Gamma}$ 
la mesure de Lebesgue sur $\Gamma_{\mathbb{R}}$ normalis\'ee de sorte \`a ce que le simplexe unit\'e soit de mesure $1$. De m\^eme, on notera 
$\tilde{\Gamma}_{\mathbb{Z}}$ le r\'eseau $\tilde{\Gamma}_{\mathbb{R}}\cap \mathfrak{X}(\tilde{T})$ de $\tilde{\Gamma}_{\mathbb{R}}$.

 On note $S$ le semi-groupe de $\mathfrak{X}(\tilde{T})$ engendr\'e par $A$.
Si $u\in S$, on note $\tilde{\Gamma}(u)$ la plus petite face de $\tilde{Q}$ contenant $u$. Si $l\geq 0$, on notera $S_l$ (resp. $\tilde{\Gamma}_l$) 
l'ensemble des $u\in S$ (resp. des $u\in \tilde{\Gamma}$) tels que $h(u)=l$. Gardons en m\'emoire que :
\begin{equation} \label{restr}
S_l=\tilde{Q}_l\cap \mathfrak{X}(\tilde{T})\text{ pour }l\gg0.
\end{equation}
En effet c'est la traduction combinatoire de la surjectivit\'e de l'application 
de restriction $H^0(\mathbb{P}((K^A)^*),\mathcal{O}(l))\to H^0(X_A,\mathcal{O}(l))$, elle-m\^eme cons\'equence du th\'eo\-r\`eme d'annulation de Serre.

On consid\`ere les $K$-espaces vectoriels :
$$C^i(A,l)=\bigoplus_{u\in S_{i+l}}\bigwedge^i(\tilde{\Gamma}(u)_{\mathbb{Z}}\otimes_{\mathbb{Z}}K).$$
Si $f=\sum c_a\chi_a\in K^A$, on d\'efinit des applications lin\'eaires
\begin{alignat*}{2}
  \partial_f:C^i(A,l)&\to C^{i+1}(A,l)\\
   (u,\omega)&\mapsto-\sum c_a(u+\chi_a,\chi_a\wedge\omega),
\end{alignat*}
de sorte que $(C^\bullet(A,l),\partial_f)$ soit un complexe. Soit de plus $e=(e(i))$ la donn\'ee de bases de chacun des $C^i(A,l)$.

\begin{thm}[\cite{GKZ} Chap.9, 2.7]\label{discri}
On suppose que $X_A$ est lisse. Alors, si $l\gg0$, on a 
\begin{equation*}
 \Delta_A(f)=\det(C^\bullet(A,l),\partial_f,e)^{(-1)^k},
\end{equation*}
o\`u chacun des deux termes est bien d\'efini \`a une constante multiplicative non nulle pr\`es.
\end{thm}

Pour la d\'efinition du d\'eterminant d'un complexe exact muni de bases, on renvoie \`a \cite{GKZ} App.A. On rappelle seulement 
ci-dessous la proposition facile \cite{GKZ} App.A 9, car nous en aurons besoin pour la suite, et que les signes sont malheureusement faux dans \cite{GKZ}.

\begin{prop}\label{basechange}

Soit $(W^{\bullet},d)$ un complexe exact, et $e=(e(i))$, $e'=(e'(i))$ deux jeux de bases. 

Soient $M(i)$ les matrices de transition : $e(i)_p=\sum_{q=1}^{\dim W_i}M(i)_{p,q}e'(i)_q$. Alors,
$$\det(W^{\bullet},d,e')=\det(W^{\bullet},d,e)\prod_i\det(M(i))^{(-1)^{i}}.$$

\end{prop}

\subsection{Action du tore}

\'Enon\c{c}ons enfin la formule pour $\Xi_A$ et d\'emontrons-la.
\begin{thm}\label{char}
\begin{equation*}
\Xi_A=\sum_{\Gamma\subset Q}(-1)^{\codim(\Gamma)}(\dim(\Gamma)+1)\int_{\Gamma}ud\mu_{\Gamma}(u).
\end{equation*}

\end{thm}

\begin{proof}[$\mathbf{Preuve}$]
 
Soit $t\in\tilde{T}$. On introduit les applications lin\'eaires :
\begin{alignat}{2}\label{gilt}
  g^{i,l}_t:C^i(A,l)&\to C^{i}(A,l)\\
   (u,\omega)&\mapsto u(t)(u,\omega)\notag.
\end{alignat}
On a un diagramme commutatif :
\begin{equation*}
\xymatrix @C=20mm {
C^i(A,l)\ar[d]_{g^{i,l}_t}\ar[r]^{\partial_f}  &C^{i+1}(A,l)\ar[d]_{g^{i+1,l}_t}
\\
C^i(A,l)\ar[r]^{\partial_{t^{-1}\cdot f}}&C^{i+1}(A,l)  }
\label{commute}
\end{equation*}
de sorte que $g_t^{\bullet,l}:(C^\bullet(A,l),\partial_f)\to(C^\bullet(A,l),\partial_{t^{-1}\cdot f})$ est un morphisme de complexes.
Par la proposition \ref{basechange}, notant $(e'(i))=(g_t^{i,l})^{-1}(e(i))$, on obtient :
\begin{alignat*}{2}
\det(C^\bullet(A,l),\partial_{t^{-1}\cdot f},e)&=\det(C^\bullet(A,l),\partial_f,e')\\
&=\det(C^\bullet(A,l),\partial_f,e)\prod_i\det(g^{i,l}_t)^{(-1)^i}.
\end{alignat*}
Par le th\'eor\`eme \ref{discri}, ceci se r\'e\'ecrit :
\begin{equation*}
 \Delta_A(t^{-1}\cdot f)^{(-1)^k}=\Delta_A(f)^{(-1)^k}\prod_i\det(g^{i,l}_t)^{(-1)^i}
\text{ pour }l\gg0.
\end{equation*}
La d\'efinition (\ref{defxi}) de $\Xi_A$ montre alors que :
\begin{equation*}
\Xi_A(t)=\prod_{i\geq0}\det(g^{i,l}_t)^{(-1)^{i+k}}\text{ pour }l\gg0.
\end{equation*}
Utilisant la d\'efinition (\ref{gilt}) de $g_t^{i,l}$, on calcule alors :
\begin{equation*}
\Xi_A(t)=\prod_{i\geq0}\prod_{u\in S_{i+l}}u(t)^{(-1)^{i+k}\dim_{\mathbb{Z}}\bigwedge^i(\tilde{\Gamma}(u)_{\mathbb{Z}})}\text{ pour }l\gg0,
\end{equation*}
ce qui se r\'e\'ecrit :
\begin{equation*}
\Xi_A=\sum_{i\geq0}\sum_{u\in S_{i+l}}(-1)^{i+k}\dim\bigwedge^i(\tilde{\Gamma}(u)_{\mathbb{R}})\cdot u\text{ pour }l\gg0.
\end{equation*}
On change alors l'ordre de sommation en regroupant les $u$ suivant la plus grande face de $\tilde{Q}$ \`a laquelle ils appartiennent. 
En prenant de plus (\ref{restr}) en compte, on obtient :
\begin{equation*}
\Xi_A=\sum_{\Gamma\subset Q}\sum_{i\geq0}(-1)^{i+k}\binom{\dim(\Gamma)+1}{i} \sum_{u\in\tilde{\Gamma}_{i+l}^0\cap \mathfrak{X}(\tilde{T})}u\text{ pour }l\gg0.
\end{equation*}
Appliquant alors les lemmes \ref{sommalt} et \ref{poly}, il vient :
\begin{alignat*}{2}
\Xi_A&=\sum_{\Gamma\subset Q}(-1)^{\dim(\Gamma)+1+k}(\dim(\Gamma)+1)\int_{\Gamma}ud\mu_{\Gamma}(u)\\
&=\sum_{\Gamma\subset Q}(-1)^{\codim(\Gamma)}(\dim(\Gamma)+1)\int_{\Gamma}ud\mu_{\Gamma}(u).
\end{alignat*}
\end{proof}

Le lemme ci-dessous est connu et est par exemple cons\'equence de \cite{BrionVergne} 4.5. Cependant, en l'absence de r\'ef\'erence o\`u il 
appara\^it sous une forme directement utilisable, j'en donne une preuve rapide utilisant la polynomialit\'e de la fonction d'Ehrhart 
(voir par exemple \cite{Ezra} 12.2).

\begin{lemme}\label{poly}
 Soit $\Gamma$ une face de $Q$. Alors, pour $l\geq0$,
la fonction $$l\mapsto\sum_{u\in\tilde{\Gamma}_{l}^0\cap \mathfrak{X}(\tilde{T})}u$$ est
polynomiale de terme dominant $\frac{l^{\dim(\Gamma)+1}}{\dim(\Gamma)!}\int_{\Gamma}ud\mu_{\Gamma}(u)$.
\end{lemme}

\begin{proof}[$\mathbf{Preuve}$]

Montrons d'abord qu'il s'agit d'un polyn\^ome de degr\'e $\leq\dim(\Gamma)+1$ pour $l\geq0$. 
En raisonnant par r\'ecurrence sur la dimension de $\Gamma$ et en appliquant une formule d'inclusion-exclusion pour 
les faces de $Q$ incluses dans $\Gamma$, on voit qu'il suffit de montrer cette propri\'et\'e pour la fonction 
$l\mapsto\sum_{u\in\tilde{\Gamma}_{l}\cap \mathfrak{X}(\tilde{T})}u$.

Il faut montrer que pour toute forme lin\'eaire enti\`ere $\psi$ sur $\mathfrak{X}(\tilde{T})$, 
l'application $P_{\psi}(l)=\sum_{u\in\tilde{\Gamma}_{l}\cap \mathfrak{X}(\tilde{T})}\psi(u)$ est polynomiale pour 
$l\geq0$. Si $\psi=h$, cette fonction est $P_{h}(l)=\sum_{u\in\tilde{\Gamma}_{l}\cap \mathfrak{X}(\tilde{T})}l=l\Card(\tilde{\Gamma}_{l}\cap \mathfrak{X}(\tilde{T}))$, 
soit $l$ fois la fonction d'Erhart de $\Gamma$, et est donc polynomiale en $l$ de degr\'e $\leq\dim(\Gamma)+1$. 

Autrement, quitte \`a ajouter \`a $\psi$ un multiple de $h$, on peut supposer $\psi$ positive sur $Q$. 
La fonction $P_\psi$ est alors la fonction d'Ehrhart du polytope auxilliaire 
$\{(x,y)\in\mathbb{R}\oplus \mathfrak{X}_{1,\mathbb{R}}|y\in Q, 0\leq x\leq\psi(y)\}$, et est donc polynomiale en $l$ de degr\'e $\leq\dim(\Gamma)+1$.

Il reste \`a calculer le coefficient dominant. C'est :
\begin{alignat*}{2}
\lim_{l\to\infty}\frac{1}{l^{\dim(\Gamma)+1}}\sum_{u\in\tilde{\Gamma}_{l}^0\cap \mathfrak{X}(\tilde{T})}u&=
\lim_{l\to\infty}\frac{1}{l^{\dim(\Gamma)}}\sum_{u\in\tilde{\Gamma}_{l}^0\cap \frac{1}{l}\mathfrak{X}(\tilde{T})}u \\
&=\frac{1}{\dim(\Gamma)!}\int_{\Gamma}ud\mu_{\Gamma}(u),
\end{alignat*}
o\`u l'on a identifi\'e dans la derni\`ere \'egalit\'e une int\'egrale et une limite de sommes de Riemann en prenant en compte la normalisation que nous 
avons avons choisie : la mesure du cube unit\'e est $\dim(\Gamma)!$ .
\end{proof}

Le lemme suivant est facile et classique :

\begin{lemme}\label{sommalt}
 Soit $N\geq0$ et $P$ un polyn\^ome de degr\'e $N$ et de coefficient dominant $a_N$. 
Alors $$\sum_{i=0}^{N}(-1)^i\binom{N}{i}P(X+i)=(-1)^NN!a_N.$$
\end{lemme}

\begin{proof}[$\mathbf{Preuve}$]
Soit $\Phi:P(X)\mapsto P(X+1)$ l'endomorphisme de l'anneau des polyn\^omes.
Par la formule du bin\^ome, $(\Id-\Phi)^N(P)=\sum_{i=0}^{N}(-1)^i\binom{N}{i}P(X+i)$. C'est alors un calcul imm\'ediat de v\'erifier que $(\Id-\Phi)$ fait 
baisser le degr\'e d'un polyn\^ome de $1$ et multiplie son coefficient dominant par l'oppos\'e de son degr\'e.
\end{proof}

\section{Degr\'es d'homog\'en\'eit\'e du discriminant}\label{part2}

Dans toute cette partie, $K$ est encore suppos\'e alg\'ebriquement clos de caract\'eristique $0$. 
On applique le th\'eor\`eme \ref{char} pour d\'emontrer le th\'eor\`eme principal \ref{princ} sous cette hypoth\`ese.

\subsection{Interpr\'etation torique du lieu discriminant}\label{intertor}

On commence par expliquer pourquoi notre probl\`eme s'inscrit dans le cadre g\'en\'eral d\'ecrit ci-dessus des vari\'et\'es duales de vari\'et\'es toriques. 

Consid\'erons le groupe ab\'elien libre de rang $c+N+1$ engendr\'e par $(Y_i)_{1\leq i\leq c}$ et $(X_j)_{0\leq j\leq N}$. On note $\alpha_i$ et $\beta_j$ les 
applications coordonn\'ees suivant $Y_i$ et $X_j$. Soit $\mathfrak{X}(\tilde{T})$ 
le sous-r\'eseau de rang $c+N$ d\'efini par l'\'equation $\sum_id_i\alpha_i=\sum_j\beta_j$. 
On note $h=\sum_i \alpha_i: \mathfrak{X}(\tilde{T})\to\mathbb{Z}$, $\mathfrak{X}(T)$ son noyau qui est un groupe 
ab\'elien libre de rang $k=c+N-1$, et $\mathfrak{X}_1=h^{-1}(1)$. 
On a par dualit\'e les morphismes de tores suivants :
\begin{equation}\label{apptor}
(K^*)^{(c+N+1)}\to\tilde{T}\to T.
\end{equation}

On introduit l'ensemble $A=\{Y_iX_0^{e_0}\ldots X_N^{e_N}\}_{1\leq i\leq c, e_j\geq0, \sum e_j=d_i}$ de $\mathfrak{X}_1$, qui l'engendre comme espace affine. 
L'espace vectoriel $K^A$ s'identifie naturellement \`a $V=\bigoplus_{1\leq i\leq c}H^0(\mathbb{P}^N,\mathcal{O}(d_i))$.

\begin{prop}\label{descridual}
On a l'\'egalit\'e suivante entre ferm\'es de $V=K^A$ : $$D=\tilde{X}^{\vee}_A.$$
\end{prop}

\begin{proof}[$\mathbf{Preuve}$]

Soit $H$ l'hyperplan de $(K^A)^*$ d'\'equation $f=\sum_i Y_iF_i=0$.
\begin{alignat}{7}
H\text{ est}&\text{ tangent \`a } \tilde{X}_A\text{ en un point de }\tilde{T}\label{dual}\\
\Leftrightarrow&\{f=0\}\text{ n'est pas un diviseur lisse de }\tilde{T}\nonumber\\
\Leftrightarrow&\{f=0\}\text{ n'est pas un diviseur lisse de }(K^*)^{(c+N+1)}\nonumber\\
\Leftrightarrow&\exists (y_1,\ldots,y_c,x_0,\ldots,x_N)\in(K^*)^{(c+N+1)} \text{ tels que }
\nonumber\\
&F_i(x_0,\ldots,x_N)=0\text{ et } \sum_i y_i\frac{\partial F_i}{\partial x_j}(x_0,\ldots,x_N)=0
\nonumber\\
\Leftrightarrow&\text{ les } F_i \text{ ont un z\'ero commun \`a coordonn\'ees non nulles en lequel leurs}\nonumber\\
&\text{ d\'eriv\'ees partielles v\'erifient une relation lin\'eaire \`a coefficients non nuls.}\label{singu}
\end{alignat}

Notons $U$ le sous-ensemble de $K^A$ constitu\'e de ces hyperplans. Par (\ref{dual}) et la  d\'efinition de la vari\'et\'e duale, 
son adh\'erence est $\tilde{X}^{\vee}_A$. Par (\ref{singu}) et le crit\`ere jacobien, $U\subset D$. Mieux : (\ref{singu}) et le lemme \ref{irred} $(i)$ et $(ii)$ 
montrent que $U$ est dense dans $D$.

Par cons\'equent, $D=\tilde{X}^{\vee}_A.$
\end{proof}

\begin{lemme}\label{irred}
La vari\'et\'e $D$ est irr\'eductible. De plus, si $(F_1,\ldots,F_c)$ est un point g\'en\'eral de $D$, les propri\'et\'es suivantes sont v\'erifi\'ees :
\begin{enumerate}[(i)]
 \item La vari\'et\'e $\{F_1=\ldots=F_c=0\}$ poss\`ede un point singulier \`a coordonn\'ees toutes non nulles.
 \item Si $I\subsetneq\{1,\ldots,c\}$, la vari\'et\'e $\{(F_i=0)_{i\in I}\}$ 
est lisse de codimension $|I|$.
\end{enumerate}
\end{lemme}

\begin{proof}[$\mathbf{Preuve}$]

Soit $Z\subset V\times\mathbb{P}^N$ le ferm\'e constitu\'e des $(F_1,\ldots,F_c,P)$ tels que la vari\'et\'e $\{F_1=\ldots=F_c=0\}$ ait 
un espace tangent de dimension $>c$ en $P$. On notera $p_1:Z\to V$ et $p_2:Z\to\mathbb{P}^N$ les deux projections. Comme $D=p_1(Z)$, 
pour montrer que $D$ est irr\'eductible, il suffit de montrer que $Z$ l'est. Pour cela, il suffit de montrer que 
pour tout $P\in \mathbb{P}^N$, $p_2^{-1}(P)$ est irr\'eductible. Utilisant l'homog\'en\'eit\'e sous $PGL_{N+1}$, 
il suffit de montrer que $p_2^{-1}([1:0:\ldots:0])$ est irr\'eductible. En \'ecrivant le crit\`ere jacobien en coordonn\'ees, 
on voit que c'est cons\'equence du fait classique que l'ensemble des matrices $(N+1)\times c$ de rang $<c$ est irr\'eductible.

Montrons $(i)$. L'ouvert $W\subset\mathbb{P}^N$ des points \`a coordonn\'ees toutes non nulles est dense. 
Son image r\'eciproque par le morphisme dominant $p_2$ est donc dense dans $Z$, et $p_1(p_2^{-1}(W))$ est dense dans $D=p_1(Z)$, ce qu'on voulait.

Montrons $(ii)$. Par Bertini, on choisit des $(F_i)_{i\in I}$ tels que $\{(F_i=0)_{i\in I}\}$ soit lisse de 
codimension $|I|$ dans $\mathbb{P}^N$, et on pose $F_i=0$ si $i\notin I$. Ceci montre qu'il existe $(F_1,\ldots,F_c)\in D$ 
tel que $\{(F_i=0)_{i\in I}\}$ soit lisse de codimension $|I|$. Comme $D$ est irr\'eductible, un point g\'en\'eral de $D$ v\'erifie cette propri\'et\'e.

\end{proof}

On d\'eduit imm\'ediatement de la proposition \ref{descridual} le corollaire suivant :

\begin{cor}\label{cridef}
On a $\codim_V(D)>1$ si et seulement si $X_A$ est d\'efective.
\end{cor}

On peut de plus relier les degr\'es d'homog\'en\'eit\'e qu'on cherche \`a calculer au caract\`ere $\Xi_A$.

\begin{cor}\label{cardeg}
On a les relations suivantes :
 \begin{enumerate}[(i)]
  \item $\deg=h(\Xi_A)$.
  \item $\deg_i=\alpha_i(\Xi_A)$.
  \item $\deg_{var}=\beta_j(\Xi_A)$.
 \end{enumerate}
\end{cor}

\begin{proof}[$\mathbf{Preuve}$]

Montrons $(ii)$, qui est la seule relation que nous utiliserons. Les autres se prouvent de mani\`ere analogue. On \'ecrit, pour $f=\sum c_a\chi_a\in K^A=V$ :
\begin{alignat*}{4}
   \lambda^{\deg_i}\Delta_A(f)&=(\rho_i(\lambda^{-1})\cdot\Delta_A)(f)&&\text{ par (\ref{defdeg})} \\
   &=\Delta_A(\rho_i(\lambda)\cdot f)&&  \\
   &=\Delta_A(\sum\lambda^{\alpha_i(\chi_a)}c_a\chi_a)&&\text{ par (\ref{action}) 
et la d\'efinition de }\alpha_i\\
    &=\lambda^{\alpha_i(\Xi_A)}\Delta_A(f)&&\text{ par (\ref{defxi}). }
\end{alignat*}
Finalement, il vient $\deg_i=\alpha_i(\Xi_A)$, ce qu'on voulait.

\end{proof}

\subsection{Le polytope $Q(c,N,(d_i)_{1\leq i\leq c})$}

Dans ce paragraphe, et dans ce paragraphe seulement, on prend temporairement des conventions l\'eg\`erement plus g\'en\'erales : 
on autorise les $d_i$ \`a \^etre des nombres r\'eels strictement positifs. On d\'efinit toujours $\mathfrak{X}_{1,\mathbb{R}}$ comme l'espace 
affine d'\'equations $\sum_i d_i\alpha_i=\sum_j\beta_j$ et $\sum_i \alpha_i=1$ dans $\mathbb{R}^{c+N+1}$. 
On pose $Q(c,N,(d_i)_{1\leq i\leq c})$ le polytope d'in\'equations $\{\alpha_i\geq0\}_{1\leq i\leq c}$ et 
$\{\beta_j\geq0\}_{0\leq j\leq N}$ dans $\mathfrak{X}_{1,\mathbb{R}}$. Il est de dimension $c+N-1$.

Si $I\subset\{1,\ldots,c\}$ et $J\subset\{0,\ldots N\}$ sont des parties non vides, le sous-ensemble
 $\Gamma_{I,J}$ de $Q$ d\'efini par les \'equations $\{\alpha_i=0\}_{i\notin I}$ et 
$\{\beta_j=0\}_{j\notin J}$ est une face de $Q$ isomorphe \`a $Q(|I|,|J|-1,(d_i)_{i\in I})$. 
De plus, toutes les faces de $Q$ sont de cette forme.

En particulier, quand les $d_i$ sont entiers, les sommets de $Q(c,N,(d_i)_{1\leq i\leq c})$ sont des \'el\'ements de $A$. 
Comme de plus $A\subset Q(c,N,(d_i)_{1\leq i\leq c})$, on voit que $Q=Q(c,N,(d_i)_{1\leq i\leq c})$. On en d\'eduit le r\'esultat suivant :
\begin{prop}\label{lisse}
 La vari\'et\'e torique $X_A$ est lisse.
\end{prop}
\begin{proof}[$\mathbf{Preuve}$]
La description explicite 
des faces de $Q$ obtenue ci-dessus permet de v\'erifier facilement le crit\`ere de lissit\'e \cite{Oda} 2.22 $(iv)$.
\end{proof}

Nous allons effectuer quelques calculs d'int\'egrales qui seront utiles par la suite. Pour les mener, nous aurons plusieurs fois besoin de la seconde partie 
du lemme ci-dessous :

\begin{lemme}\label{polyn}
Soit $P\in \mathbb{Z}[d_1,\ldots,d_c][X]$. On introduit : $$R=\sum_{l=1}^{c}\frac{P(d_l)}{\prod_{l'\neq l}(d_l-d_{l'})}.$$
\begin{enumerate}[(i)]
 \item $R\in \mathbb{Z}[d_1,\ldots,d_c]$.
 \item Si $P$ est de degr\'e $\leq c-2$ en $X$, $R=0$.
\end{enumerate}
\end{lemme}

\begin{proof}[$\mathbf{Preuve}$]
On introduit $R'=\sum_{l=1}^{c}\frac{P(X_l)}{\prod_{l'\neq l}(X_l-X_{l'})}\in \mathbb{Z}[d_1,\ldots,d_c](X_1,\ldots,X_c)$. 
Multipliant par $X_{l_1}-X_{l_2}$, puis sp\'ecialisant en $X_{l_1}=X_{l_2}$, on obtient $0$. Par cons\'equent, $R'\in\mathbb{Z}[d_1,\ldots,d_c][X_1,\ldots,X_c]$.

En faisant $X_l=d_l$, on montre que $R\in \mathbb{Z}[d_1,\ldots,d_c]$. 

Si de plus $P$ est de degr\'e $\leq c-2$ en $X$, $R'$ est un polyn\^ome de degr\'e $<0$ en les $X_i$, et est donc nul. En faisant $X_l=d_l$, cela implique $R=0$.
\end{proof}

Calculons tout d'abord le volume du polytope $Q(c,N,(d_i)_{1\leq i\leq c})$. On rappelle notre convention d'attribuer une mesure $1$ au simplexe unit\'e.

\begin{prop}\label{vol}

$$\mu(Q(c,N,(d_i)_{1\leq i\leq c}))=\sum_{l=1}^c\frac{d_l^{c+N-1}}{\prod_
{l'\neq l}(d_l-d_{l'})}.$$

\end{prop}

\begin{proof}[$\mathbf{Preuve}$]

On proc\`ede par r\'ecurrence sur $c$. Pour $c=1$, c'est la formule du volume du simplexe de c\^ot\'e $d_1$. 
Si $c\geq2$, on applique Fubini en remarquant que l'image de $Q(c-1,N,(s d_i+(1-s)d_c)_{1\leq i\leq c-1})$ 
par l'application $$(y_1,\ldots,y_c,x_0,\ldots,x_N)\mapsto((1-s)y_1,\ldots,(1-s)y_c,x_0,\ldots,x_N)$$ 
est $Q(c,N,(d_i)_{1\leq i\leq c})\cap\{\alpha_c=s\}$. 
Entre les espaces affines qui nous int\'eressent cette application est de d\'eterminant $(1-s)^{c-2}$. 
On peut alors appliquer l'hypoth\`ese de r\'ecurrence. Il vient :
\begin{alignat*}{3}
  \mu&(Q(c,N,(d_i)_{1\leq i\leq c}))\\
&=(c+N-1)\int_0^1(1-s)^{c-2}\mu(Q(c-1,N,(s d_i+(1-s)d_c)_{1\leq i\leq c-1}))ds\\
   &=\sum_{l=1}^{c-1}\int_0^1\frac{((1-s)d_c+s d_l)^{c+N-2}}{\prod_
{l'\neq l,c}(d_l-d_{l'})}ds\\
 &=\sum_{l=1}^{c}\frac{1}{\prod_
{l'\neq l}(d_l-d_{l'})}(d_l^{c+N-1}-d_c^{c+N-1})\\
&=\sum_{l=1}^{c}\frac{d_l^{c+N-1}}{\prod_
{l'\neq l}(d_l-d_{l'})}\text{       par le lemme \ref{polyn} $(ii)$. }
\end{alignat*}
\end{proof}

Enfin, nous utiliserons dans le paragraphe suivant le calcul de l'int\'egrale ci-dessous :

\begin{prop}\label{inte}

$$\int_{Q(c,N,(d_i)_{1\leq i\leq c})}\alpha_c(u)d\mu(u)=\frac{1}{(c+N)}\sum_{l=1}^{c}\frac{1}{\prod_
{l'\neq l}(d_l-d_{l'})}\left(\frac{d_c^{c+N}-d_l^{c+N}}{d_c-d_l}\right).$$

\end{prop}

\begin{proof}[$\mathbf{Preuve}$]

On applique Fubini comme dans le calcul pr\'ec\'edent.
\begin{alignat*}{3}
  \int&_{Q(c,N,(d_i)_{1\leq i\leq c})}\alpha_c(u)d\mu(u)\\
&=(c+N-1)\int_0^1s(1-s)^{c-2}\mu(Q(c-1,N,(s d_i+(1-s)d_c)_{1\leq i\leq c-1}))ds\\
   &=(c+N-1)\sum_{l=1}^{c-1}\int_0^1s\frac{((1-s)d_c+s d_l)^{c+N-2}}{\prod_
{l'\neq l,c}(d_l-d_{l'})}ds\\
 &=\sum_{l=1}^{c-1}\left[\int_0^1\frac{((1-s)d_c+sd_l)^{c+N-1}}{\prod_
{l'\neq l}(d_l-d_{l'})}ds-\frac{d_c^{c+N-1}}{\prod_
{l'\neq l}(d_l-d_{l'})}\right],
\end{alignat*}
o\`u l'on a int\'egr\'e par parties. Calculant l'int\'egrale du terme de gauche, et appliquant le lemme \ref{polyn} $(ii)$ pour sommer le terme de droite, on obtient :
\begin{alignat*}{3}
 \int&_{Q(c,N,(d_i)_{1\leq i\leq c})}\alpha_c(u)d\mu(u)\\
 &=\frac{1}{(c+N)}\left[\sum_{l=1}^{c-1}\frac{1}{\prod_
{l'\neq l}(d_l-d_{l'})}\left(\frac{d_c^{c+N}-d_l^{c+N}}{d_c-d_l}\right)+\frac{(c+N)d_c^{c+N-1}}{\prod_
{l'\neq c}(d_c-d_{l'})}\right]\\
&=\frac{1}{(c+N)}\sum_{l=1}^{c}\frac{1}{\prod_
{l'\neq l}(d_l-d_{l'})}\left(\frac{d_c^{c+N}-d_l^{c+N}}{d_c-d_l}\right).
\end{alignat*}

\end{proof}

\subsection{Homog\'en\'eit\'e en les \'equations}

Montrons la premi\`ere partie du th\'eor\`eme \ref{princ}. Par sym\'etrie, on peut supposer $i=c$.

\begin{prop} \label{premprinc}
   $$\deg_c =d_1\ldots d_{c-1}\sum_{l=1}^{c}\frac{1}{\prod_
{l'\neq l}(e_l-e_{l'})}\left(\frac{e_c^{N+1}-e_l^{N+1}}{e_c-e_l}\right).$$

\end{prop}

\begin{proof}[$\mathbf{Preuve}$]

On utilise la relation \ref{cardeg} $(ii)$, et la formule \ref{char} pour $\Xi_A$ qui s'applique car $X_A$ est lisse par la proposition \ref{lisse} :
$$\deg_c=\sum_{\Gamma\subset Q(c,N,(d_i)_{1\leq i\leq c})}(-1)^{\codim(\Gamma)}(\dim(\Gamma)+1)\int_{\Gamma}\alpha_c(u)d\mu_{\Gamma}(u).$$
Les faces de $Q(c,N,(d_i)_{1\leq i\leq c})$ sont les $\Gamma_{I,J}$. L'int\'egrale qui intervient est nulle si $c\notin I$ car $\alpha_c$ 
s'annule alors identiquement sur $\Gamma_{I,J}$. Si $c\in I$, on reconna\^it l'int\'egrale calcul\'ee en \ref{inte}. Il vient :
$$\deg_c=\sum\limits_{\substack{I\subset\{1,\ldots,c-1\}\\ \varnothing\neq J\subset\{0,\ldots,N\}}}\sum_{l\in I\cup\{c\}}\frac{(-1)^{c+N-|I|-|J|}}{\prod\limits_{\substack{l'\in I\cup\{c\}\\l'\neq l}}
(d_l-d_{l'})}\left(\frac{d_c^{|I|+|J|}-d_l^{|I|+|J|}}{d_c-d_l}\right).$$
En param\'etrant $J$ par $j=|J|$, et en remarquant que le terme $j=0$ dans la somme ci-dessous est nul par le lemme \ref{polyn} $(ii)$, on obtient :
$$\deg_c=\sum_{j=0}^{N+1}\tbinom{N+1}{j}\sum_{I\subset\{1,\ldots,c-1\}}\sum_{l\in I\cup\{c\}}\frac{(-1)^{c+N-|I|-j}}{\prod\limits_{\substack{l'\in I\cup\{c\}\\l'\neq l}}
(d_l-d_{l'})}\left(\frac{d_c^{|I|+j}-d_l^{|I|+j}}{d_c-d_l}\right).$$
Appliquons la formule du bin\^ome.
$$\deg_c=\sum_{I\subset\{1,\ldots,c-1\}}\sum_{l\in I\cup\{c\}}\frac{(-1)^{c-|I|-1}}{\prod\limits_{\substack{l'\in I\cup\{c\}\\l'\neq l}}
(d_l-d_{l'})}\left(\frac{d_c^{|I|}e_c^{N+1}-d_l^{|I|}e_l^{N+1}}{d_c-d_l}\right).$$
\'Echangeons alors les sommations sur $I$ et sur $l$. On note $\bar{I}$ le compl\'ementaire de $I$ dans $\{1,\ldots,c-1\}$, 
et on remarque que dans la somme ci-dessous la contribution des termes pour lesquels $l\in\bar{I}$ est nulle.
$$\deg_c=\sum_{l=1}^c\frac{1}{\prod_
{l'\neq l}(d_l-d_{l'})}\frac{1}{d_c-d_l}M_l, \text{ o\`u}$$
$$M_l=\sum_{\bar{I}\subset\{1,\ldots,c-1\}}(-1)^{|\bar{I}|}
(d_c^{c-1-|\bar{I}|}e_c^{N+1}-d_l^{c-1-|\bar{I}|}e_l^{N+1})\prod_{l'\in\bar{I}}(d_l-d_{l'}).$$
Calculons $M_l$. On commence par d\'evelopper le produit pour obtenir :
$$M_l=\sum\limits_{\substack{\bar{I}\subset\{1,\ldots,c-1\}\\H\subset\bar{I}}}
(d_c^{c-1-|\bar{I}|}e_c^{N+1}-d_l^{c-1-|\bar{I}|}e_l^{N+1})(-1)^{|H|+|\bar{I}|}d_l^{|\bar{I}|-|H|}\prod_{l'\in H}d_{l'}.$$
En sommant d'abord sur $H$, puis sur le cardinal $i=|\bar{I}|-|H|$, on obtient pour $M_l$ l'expression suivante :
$$\sum_{H\subset\{1,\ldots,c-1\}}\prod_{l'\in H}d_{l'}\sum_{i=0}^{c-|H|-1}\tbinom{c-|H|-1}{i}
(-d_l)^i(d_c^{^{c-1-i-|H|}}e_c^{N+1}-d_l^{^{c-1-i-|H|}}e_l^{N+1}).$$
Appliquons \`a nouveau la formule du bin\^ome.
$$M_l=\sum_{H\subset\{1,\ldots,c-1\}}\prod_{l'\in H}d_{l'}\left((d_c-d_l)^{c-1-|H|}e_c^{N+1}-(d_l-d_l)^{c-1-|H|}e_l^{N+1}\right).$$
Le terme de droite se calcule en remarquant que $(d_l-d_l)^{c-1-|H|}$ est non nul seulement si $|H|=c-1$, c'est-\`a-dire si $H=\{1,\ldots,c-1\}$. 
Quant au terme de gauche, on peut le factoriser ais\'ement. Il reste :
$$M_l=e_c^{N+1}\prod_{l'=1}^{c-1}(d_{l'}+d_c-d_l)-e_l^{N+1}\prod_{l'=1}^{c-1}d_{l'}.$$
Reprenant le calcul de $\deg_c$, on voit que :
$$\deg_c=\sum_{l=1}^c\frac{1}{\prod_
{l'\neq l}(d_l-d_{l'})}\left(\frac{e_c^{N+1}\prod_{l'=1}^{c-1}(d_{l'}+d_c-d_l)-e_l^{N+1}\prod_{l'=1}^{c-1}d_{l'} }{d_c-d_l}\right).$$
Par le lemme \ref{polyn} $(ii)$, comme $\frac{e_c^{N+1}\prod_{l'=1}^{c-1}(d_{l'}+d_c-d_l)-e_c^{N+1}\prod_{l'=1}^{c-1}d_{l'}}{d_c-d_l}$ 
est un polyn\^ome de degr\'e $c-2$ en $d_l$, on calcule pour conclure :
\begin{alignat*}{2}
 \deg_c&=\sum_{l=1}^c\frac{1}{\prod_
{l'\neq l}(d_l-d_{l'})}\left(\frac{e_c^{N+1}\prod_{l'=1}^{c-1}d_{l'}-e_l^{N+1}\prod_{l'=1}^{c-1}d_{l'} }{d_c-d_l}\right)\\
 &=d_1\ldots d_{c-1}\sum_{l=1}^{c}\frac{1}{\prod_
{l'\neq l}(e_l-e_{l'})}\left(\frac{e_c^{N+1}-e_l^{N+1}}{e_c-e_l}\right).
\end{alignat*}
\end{proof}

\subsection{Homog\'en\'eit\'e en les variables}

Montrons la seconde partie du th\'eor\`eme \ref{princ}.

\begin{prop} \label{deuxprinc}
  $$\deg_{var}=d_1\ldots d_c\sum_{l=1}^{c}\frac{e_l^N}{\prod_{l'\neq l}(e_l-e_{l'})}.$$
\end{prop}

\begin{proof}[$\mathbf{Preuve}$]

On pourrait proc\'eder par calcul direct comme en \ref{premprinc}. On va plut\^ot profiter du calcul d\'ej\`a effectu\'e en \ref{premprinc} 
et de la relation (\ref{degvar}). Il vient :
\begin{alignat*}{2}
  \deg_{var}&=\frac{1}{N+1}\sum_{i=1}^c d_i\deg_i\\
   &=\frac{d_1\ldots d_c}{N+1}\sum_{i=1}^c\sum_{l=1}^c \frac{1}{\prod_
{l'\neq l}(e_l-e_{l'})}\left(\frac{e_i^{N+1}-e_l^{N+1}}{e_i-e_l}\right).
\end{alignat*}
Utilisant alors le calcul report\'e dans le lemme \ref{petitcalcul}, et vu la remarque \ref{remconv}, on obtient :
\begin{alignat*}{2}
  \deg_{var}&=\frac{d_1\ldots d_c}{N+1}\sum_{l=1}^c \frac{1}{\prod_
{l'\neq l}(e_l-e_{l'})}\left(\frac{e_l^{N+1}-e_l^{N+1}}{e_l-e_l}\right)\\
   &=d_1\ldots d_c\sum_{l=1}^c \frac{e_l^N}{\prod_
{l'\neq l}(e_l-e_{l'})}.
\end{alignat*}
\end{proof}

\begin{lemme}\label{petitcalcul}
$$\sum_{l=1}^c\sum\limits_{\substack{i=1\\i\neq l}}^c \frac{1}{\prod_
{l'\neq l}(e_l-e_{l'})}\left(\frac{e_i^{N+1}-e_l^{N+1}}{e_i-e_l}\right)=0.$$
\end{lemme}

\begin{proof}[$\mathbf{Preuve}$]

Fixons $l$ et introduisons $\Phi_l=\sum_{i\neq l}\frac{1}{\prod_
{i'\neq i}(e_i-e_{i'})}\prod_
{l'\neq i,l}(e_l-e_{l'})$, qu'on consid\`ere comme une fraction rationnelle en $e_l$ \`a coefficients dans $\mathbb{Q}((e_i)_{i\neq l})$. 
\'Ecrivons sa d\'ecomposition en \'el\'ements simples  $\Phi_l=\sum_{i\neq l} \frac{f_i}{e_l-e_i}$. En multipliant $\Phi_l$ par $(e_l-e_i)$, 
et en substituant $e_l=e_i$ dans l'expression obtenue, on calcule $f_i=-1$. On a montr\'e :
$$\sum_{i\neq l}\frac{1}{\prod_
{i'\neq i}(e_i-e_{i'})}\frac{1}{e_l-e_i}=-\frac{1}{\prod_
{l'\neq l}(e_l-e_{l'})}\sum_{i\neq l}\frac{1}{e_l-e_i}.$$

Multiplions cette identit\'e par $e_l^{N+1}$, sommons sur $l$, puis \'echangeons dans le terme de gauche le r\^ole des variables muettes $i$ et $l$ pour obtenir :

$$\sum_{l=1}^c\sum\limits_{\substack{i=1\\i\neq l}}^c \frac{1}{\prod_
{l'\neq l}(e_l-e_{l'})}\frac{e_i^{N+1}}{e_i-e_l}=-\sum_{l=1}^c\sum\limits_{\substack{i=1\\i\neq l}}^c \frac{1}{\prod_
{l'\neq l}(e_l-e_{l'})}\frac{e_l^{N+1}}{e_l-e_i}.$$

Faisons tout passer dans le terme de gauche : le lemme est d\'emontr\'e.
\end{proof}

\section{Caract\'eristique finie}

On explique dans cette partie comment modifier la preuve propos\'ee ci-dessus pour d\'emontrer le th\'eor\`eme \ref{princ} quand $K$ est 
de caract\'eristique finie. 
On en d\'eduit alors une preuve du th\'eor\`eme \ref{princ2}.

\subsection{\'{E}quation de la duale}

On conserve les notations de la partie \ref{part1}.

 Le th\'eor\`eme \cite{GKZ} Chap.9, 2.7, que nous  que nous avons \'enonc\'e en \ref{discri} d\'ecrivait l'\'equation de la vari\'et\'e duale
d'une vari\'et\'e torique lisse $X_A\subset\mathbb{P}((K^A)^*)$. 
Il ne vaut tel quel qu'en caract\'eristique $0$, et son \'enonc\'e doit \^etre modifi\'e en g\'en\'eral.
\`{A} cet effet, on introduit les notations suivantes.

Soit $W_{X_A}\subset\mathbb{P}((K^A)^*)\times\mathbb{P}(K^A)$ la vari\'et\'e d'incidence de $X_A$, c'est-\`a-dire l'adh\'erence de l'ensemble des couples
$(x,H)\in\mathbb{P}((K^A)^*)\times\mathbb{P}(K^A)$ tels que $H$ soit tangent en le point lisse $x$ de $X_A$. 
Notons $p_1$ et $p_2$ les projections de $W_{X_A}$ sur $X_A$ et $X_A^\vee=p_2(W_{X_A})$ respectivement. 
On munit $W_{X_A}$ et $X_A^{\vee}$ de leur structure r\'eduite.

Si $X_A$ n'est pas d\'efective, c'est-\`a-dire si $X_A^{\vee}$ est une hypersurface de $\mathbb{P}(K^A)$,
on note $\mu$ le degr\'e de l'application g\'en\'eriquement finie $p_2:W_{X_A}\to X_A^{\vee}$. Le th\'eor\`eme 
\ref{discri} admet alors la g\'en\'eralisation suivante :

\begin{thm}\label{discribis}
On suppose que $X_A$ est lisse, et $l\gg0$. 
\begin{enumerate}[(i)]
 \item 
Si $X_A$ n'est pas d\'efective, 
\begin{equation*}
 \Delta_A(f)^{\mu}=\det(C^\bullet(A,l),\partial_f,e)^{(-1)^k},
\end{equation*}
o\`u chacun des deux termes est bien d\'efini \`a une constante multiplicative non nulle pr\`es.
\item
Si $X_A$ est d\'efective, rappelons que $\Delta_A(f)=1$ par convention. Alors $$\Delta_A(f)=\det(C^\bullet(A,l),\partial_f,e)^{(-1)^k},$$ 
au sens o\`u le terme de droite est \'egalement une constante non nulle.
\end{enumerate}
\end{thm}

\begin{proof}[$\mathbf{Preuve}$]
 
La preuve de \cite{GKZ} Chap.9, 2.7 ne n\'ecessite qu'une modification mineure, que l'on va d\'ecrire. 
Cette preuve fait appel au th\'eor\`eme \cite{GKZ} Chap.2, 2.5. 
Au cours de la preuve de cet autre th\'eor\`eme, on utilise (page 59) le fait que si $X_A$ n'est pas d\'efective, $p_2:W_{X_A}\to X_A^{\vee}$ est birationnelle. 
En caract\'eristique $0$,
c'est une cons\'equence du th\'eor\`eme de r\'eflexivit\'e. Cependant, en caract\'eristique finie, $p_2:W_{X_A}\to X_A^{\vee}$ est seulement g\'en\'eriquement finie,
de degr\'e $\mu$. 

En prenant cette modification en compte, et en adaptant les arguments de mani\`ere \'evidente, on prouve le th\'eor\`eme.
\end{proof}

Tous les autres arguments que nous avons utilis\'es sont encore valables. Nous utiliserons librement les r\'esultats d\'ej\`a obtenus, notamment l'identification
de $D$ \`a la vari\'et\'e duale d'une vari\'et\'e torique explicite (proposition \ref{descridual}).

\subsection{Calcul du degr\'e $\mu$}
 Pour appliquer le th\'eor\`eme \ref{discribis}, il faut calculer la quantit\'e $\mu$ 
dans les cas o\`u $X_A$ n'est pas d\'efective : c'est le but de la proposition \ref{mu}.
On conserve les notations du paragraphe \ref{intertor}.

\begin{prop}\label{mu}
Supposons qu'on n'ait pas $d_1=\ldots=d_c=1$ et $c<N+1$.

Alors $X_A$ n'est pas d\'efective.
Si $K$ n'est pas de caract\'eristique $2$ ou si $n$ est impair, $\mu=1$. Sinon, $\mu=2$.
\end{prop}

On commence par montrer plusieurs lemmes.
Les arguments qui suivent sont l\'eg\`erement alourdis par le fait qu'il faut 
manipuler avec pr\'ecaution les points doubles ordinaires en caract\'eristique $2$.

\begin{lemme}\label{icgood}
 
\begin{enumerate}[(i)]
 \item 
Supposons que $c=N+1$. Alors pour $(F_1,\ldots, F_c)\in D$ g\'en\'eral, $\{F_1=\ldots=F_c=0\}$ est un unique point r\'eduit.

\item
Supposons que $c<N+1$ et qu'il existe $i$ tel que $d_i\geq2$. Alors pour $(F_1,\ldots, F_c)\in D$ g\'en\'eral, $\{F_1=\ldots=F_c=0\}$ est de codimension $c$ et a
un unique point singulier, 
qui est un point double ordinaire.

\item 
Supposons qu'on n'ait pas $d_1=\ldots=d_c=1$ et $c<N+1$. Alors pour $(F_1,\ldots, F_c)\in D$ g\'en\'eral, les 
diff\'erentielles des $F_i$ sont li\'ees en un unique point de $\{F_1=\ldots=F_c=0\}$, et ce par une unique relation.

\end{enumerate}
\end{lemme}

\begin{proof}[$\mathbf{Preuve}$]
\begin{enumerate}[(i)]
 \item
Cette propri\'et\'e est ouverte dans $D$ qui est irr\'eductible par le lemme \ref{irred} : il suffit donc d'exhiber un jeu d'\'equations v\'erifiant cette propri\'et\'e. 
Par Bertini, on choisit $F_1,\ldots, F_{c-1}$ g\'en\'eraux de sorte que le sch\'ema $\{F_1=\ldots=F_{c-1}=0\}$ soit r\'eunion de points r\'eduits. On choisit alors $F_c$ 
de sorte \`a ce qu'il passe par un de ces points et \'evite les autres.

\item
Par description de la d\'eformation verselle d'un point double ordinaire (voir \cite{SGA7} Exp. XV Prop. 1.3.1), cette propri\'et\'e est ouverte dans $D$.
Comme $D$ est irr\'eductible par le lemme \ref{irred}, 
il suffit donc d'exhiber un jeu d'\'equations v\'erifiant cette propri\'et\'e. 
On consid\`ere le syst\`eme lin\'eaire constitu\'e des $F_i$ passant par $P=[0:\ldots:0:1]$, y \'etant singuliers et dont les termes d'ordre $2$ sont un multiple
d'une forme quadratique ordinaire fix\'ee. Le th\'eor\`eme de Bertini assure que le membre 
g\'en\'eral de ce syst\`eme lin\'eaire a pour unique point singulier $P$ ; c'est
un point double ordinaire.
On prend alors $F_1,\ldots,\hat{F_i},\ldots, F_c$ g\'en\'erales
passant par $P$. Le th\'eor\`eme de Bertini assure que $(F_1,\ldots, F_c)$ convient.

\item Les deux premiers points permettent de d\'ecrire, pour $(F_1,\ldots, F_c)\in D$ g\'en\'eral, les dimensions des espaces tangents de $\{F_1=\ldots=F_c=0\}$.
On en d\'eduit le r\'esultat.

\end{enumerate}

\end{proof}

\begin{lemme}\label{nondef}
Supposons qu'on n'ait pas $d_1=\ldots=d_c=1$ et $c<N+1$. Alors :
\begin{enumerate}[(i)]
 \item
La vari\'et\'e  $X_A$ n'est pas d\'efective.
 \item
De plus, si $H$ est un \'el\'ement g\'en\'eral de $D=\tilde{X}_A^\vee$ vu comme un hyperplan de $\mathbb{P}((K^A)^*)$, 
$H$ est tangent \`a $X_A$ en un unique point.

\end{enumerate}
\end{lemme}

\begin{proof}[$\mathbf{Preuve}$]
\begin{enumerate}[(i)]
 \item
Choisissons $H\in D$ g\'en\'eral comme dans les lemmes \ref{icgood} $(iii)$ et \ref{irred} $(i)$ et $(ii)$. L'\'equivalence entre (\ref{dual}) et (\ref{singu}) dans la 
preuve de \ref{descridual} montre alors qu'il existe un unique point $t$ de $T$ en lequel $H$ est tangent \`a $X_A$. En particulier, $t$ est isol\'e
dans $p_2^{-1}(H)$. Comme $W_{X_A}$ est irr\'eductible, cela implique que $p_2$ est g\'en\'eriquement finie. Ainsi, $X_A$ n'est pas d\'efective.

 \item
Soit $H\in D$ g\'en\'eral comme au point pr\'ec\'edent. Choisissons-le de plus hors de $X_A^\vee\setminus p_2(p_1^{-1}(X_A\setminus T))$, qui est
un ferm\'e strict car $p_2$ est g\'en\'eriquement finie. L'hyperplan $H$ n'est tangent \`a $X_A$ qu'en des points de $T$, et est tangent \`a $T$ en un unique point. Cela 
conclut.

\end{enumerate}

\end{proof}

\begin{lemme}\label{ptdbl}
Supposons qu'on n'ait pas $d_1=\ldots=d_c=1$ et $c<N+1$.
Alors si $H$ est un \'el\'ement g\'en\'eral de $D=\tilde{X}_A^\vee$ vu comme un hyperplan de $\mathbb{P}((K^A)^*)$, 
$X_A\cap H$ a un unique point singulier qui est un point double ordinaire.
\end{lemme}

\begin{proof}[$\mathbf{Preuve}$]
On choisit $H=(F_1,\ldots,F_c)$ g\'en\'eral comme dans le lemme pr\'ec\'edent, et comme dans le lemme \ref{icgood} $(i)$ (resp. $(ii)$).
On notera $Z=\{F_1=\ldots=F_c=0\}\subset\mathbb{P}^N$ et $\tilde{Z}\subset\mathbb{A}^{N+1}$ le c\^one affine sur $Z$.

Les choix faits montrent que $H$ est tangent \`a $X_A$ en un unique point $t\in T$. La preuve de la proposition \ref{descridual} montre que 
si $t'=(y_1,\ldots,y_c,x_0,\ldots,x_N)\in(K^*)^{c+N+1}$ est un ant\'ec\'edent de $t$ par l'application
(\ref{apptor}), $x=[x_0,\ldots,x_N]$ est l'unique point de $Z$ (resp. l'unique point singulier de $Z$), 
que c'est un point r\'eduit (resp. un point double ordinaire), et que l'unique relation entre les diff\'erentielles
des $F_i$ en $\tilde{x}=(x_0,\ldots,x_N)$ est donn\'ee par $\sum_i y_i\frac{\partial F_i}{\partial X_j}(\tilde{x})=0$, $0\leq j\leq N$.

Soit $q$ la forme quadratique induite par $H$ sur $T_tT$. On raisonne par l'absurde en la supposant non ordinaire : 
il existe $w\in T_tT$ non nul tel que $w\in\rad(q)$ et $q(w)=0$.
On note $q'$ la forme quadratique que $q$ induit sur $T_{t'}(K^*)^{c+N+1}$ via l'application (\ref{apptor}). Le noyau de la surjection
$T_{t'}(K^*)^{c+N+1}\to T_tT$ est engendr\'e
par $w'_1=(y_1,\ldots,y_c,0,\ldots,0)$ et $w'_2=(-d_1y_1,\ldots,-d_cy_c,x_0,\ldots,x_N)$. 
Notons $w'=(u_1,\ldots,u_c,v_0,\ldots,v_N)$ un ant\'ec\'edent de $w$, de sorte que $w'\in\rad(q')$, $q'(w')=0$ et 
$w'\notin\langle w'_1,w'_2\rangle$. 

La condition $w'\in\rad(q')$ signifie que $w'$ appartient au noyau de la matrice de la forme bilin\'eaire associ\'ee
\`a $q'$, c'est-\`a-dire au noyau de la Hessienne de $H$ en $t'$. C'est un syst\`eme d'\'equations qui s'\'ecrit :
\begin{alignat}{2}
 \sum_j v_j\frac{\partial F_i}{\partial X_j}(\tilde{x})&=0,&& 1\leq i\leq c.             \label{Hess1} \\
\sum_jv_j\frac{\partial^2\sum_i y_iF_i}{\partial X_j\partial X_k}(\tilde{x})
&=\sum_i u_i\frac{\partial  F_i}{\partial X_k}(\tilde{x}), \text{ } &&0\leq k\leq N.\label{Hess2}      
\end{alignat}

Consid\'erons $\tilde{v}=(v_0,\ldots,v_N)$ comme un vecteur
tangent \`a $\mathbb{A}^{N+1}$ en $\tilde{x}$. L'\'equation (\ref{Hess1}) montre que le vecteur $\tilde{v}$ appartient \`a $T_{\tilde{x}}\tilde{Z}$.
Montrons qu'il est non radial. Si c'\'etait le cas, on pourrait, quitte \`a retrancher \`a $w'$ un multiple de $w'_2$, le
supposer nul. L'\'equation (\ref{Hess2}) fournit alors la relation
 $\sum_i u_i\frac{\partial F_i}{\partial X_k}(\tilde{x})=0$, $0\leq k\leq N$ entre les diff\'erentielles des $F_i$ en $\tilde{x}$. Ce n'est possible par
hypoth\`ese que si $w'$ est proportionnel \`a $w'_1$, ce qui contredit $w'\notin\langle w'_1,w'_2\rangle$. Ainsi $\tilde{v}$ n'est pas radial.

On distingue alors deux cas.
\begin{enumerate}[(i)]
 \item 
Supposons que $c=N+1$. Un vecteur $v\in T_xZ$ se relevant en $\tilde{v}$ est alors un \'el\'ement non nul de $T_x Z$, et
$Z$ n'est donc pas un point r\'eduit. C'est la contradiction recherch\'ee.
\item
Supposons que $c<N+1$ et qu'il existe $i$ tel que $d_i\geq2$.

Vu l'unique relation liant les diff\'erentielles en $\tilde{x}$ des $F_i$, la forme quadratique qui est l'\'equation dans $T_{\tilde{x}}\tilde{Z}$
du c\^one tangent \`a $\tilde{Z}$ en $\tilde{x}$ est la restriction \`a $T_{\tilde{x}}\tilde{Z}$ de la forme quadratique induite par
les termes d'ordre deux de $\sum y_iF_i$. On note $\tilde{Q}$ cette forme quadratique.

Montrons que $\tilde{v}\in\rad(\tilde{Q})$. L'\'equation (\ref{Hess2}) signifie que si un vecteur est orthogonal \`a 
$(\frac{\partial}{\partial X_k}(\sum_i u_i F_i)(\tilde{x}))_{0\leq k\leq N}$ pour le produit scalaire usuel, il est automatiquement orthogonal \`a $\tilde{v}$
pour la forme bilin\'eaire associ\'ee \`a $\tilde{Q}$. Or tous les vecteurs de $T_{\tilde{x}}\tilde{Z}$ sont orthogonaux \`a
$(\frac{\partial}{\partial X_k}(\sum_i u_i F_i)(\tilde{x}))_{0\leq k\leq N}$, vu comme un gradient.
On a bien montr\'e $\tilde{v}\in\rad(\tilde{Q})$.

On v\'erifie ensuite que l'\'equation $q'(w')=0$ est la m\^eme \'equation que $\tilde{Q}(\tilde{v})=0$. Soit alors $v\in T_xZ$ se relevant en $\tilde{v}$, et
$Q$ la forme quadratique sur $T_xZ$, \'equation du c\^one tangent \`a $Z$ en $x$, induisant $\tilde{Q}$ sur $T_{\tilde{x}}\tilde{Z}$. On
a montr\'e que $v$ est un \'el\'ement non nul de $\rad(Q)$ sur lequel $Q$ s'annule. La forme quadratique $Q$ n'est donc pas ordinaire et 
$x$ ne peut \^etre un point double ordinaire de $Z$. 
C'est absurde.
\end{enumerate}
\end{proof}

On peut alors prouver la proposition \ref{mu} :
\begin{proof}[$\mathbf{Preuve \text{ }de \text{ }la\text{ } proposition\text{ }\ref{mu}}$]
La vari\'et\'e $X_A$ n'est pas d\'efective par \ref{nondef} $(i)$, de sorte que $\mu$ est bien d\'efini.

Soit $H$ un point g\'en\'eral de $X_A^\vee$. Soit $p_2^{-1}(H)\subset X_A$ le lieu sch\'ematique le long duquel $H$ est tangent \`a $X_A$ ; 
par d\'efinition de $\mu$, on a $\mu=\lon(p_2^{-1}(H))$. 
Par \cite{Kleiman} I (8) et (9), $p_2^{-1}(H)=\Sing(X_A\cap H)$ o\`u le lieu singulier de $X_A\cap H$ est muni 
de la structure sch\'ematique donn\'ee par le $(k-1)$-i\`eme
id\'eal de Fitting du faisceau des diff\'erentielles de K\"ahler. 

Comme $H$ est choisi g\'en\'eral, par le lemme \ref{ptdbl}, $X_A\cap H$ a un unique point singulier qui est un point double ordinaire. 
Pour calculer $\Sing(X_A\cap H)$, on peut travailler dans le compl\'et\'e de $X_A\cap H$ en 
ce point. Par \cite{SGA7} Exp. XV Th. 1.2.6, celui-ci est isomorphe au le lieu des z\'eros dans $K[[x_1,\ldots,x_{k}]]$ de la forme 
quadratique ordinaire canonique $x_1x_2+\ldots+x_{k-1}x_{k}$
si $k$ est pair ou $x_1x_2+\ldots+x_{k-2}x_{k-1}+x_{k}^2$ si $k$ est impair. 

Sur ces \'equations, il est facile de calculer le $(k-1)$-i\`eme
id\'eal de Fitting du faisceau des diff\'erentielles : c'est $\langle x_1,\ldots, x_{k}\rangle$ sauf si $k$ est impair et 
$K$ est de caract\'eristique $2$ auquel cas c'est $\langle x_1,\ldots,x_{k-1},x_{k}^2\rangle$. Dans le premier cas, le sous-sch\'ema qu'il 
d\'efinit est un point r\'eduit et $\mu=\lon(p_2^{-1}(H))=1$. Dans le second cas, il d\'efinit un sous-sch\'ema de longueur $2$ et $\mu=\lon(p_2^{-1}(H))=2$.
Comme $n=N-c$ est de parit\'e oppos\'ee \`a $k=c+N-1$, la proposition est
d\'emontr\'ee.
\end{proof}

\subsection{Preuve du th\'eor\`eme principal}

On commence par montrer que dans les cas non trait\'es par le lemme \ref{nondef}, $X_A$ est d\'efective.

\begin{lemme}\label{def}
Supposons $d_1=\ldots=d_c=1$ et $c<N+1$. Alors $X_A$ est d\'efective. 
\end{lemme}

\begin{proof}[$\mathbf{Preuve}$]
Les $F_i$ sont des formes lin\'eaires. La sous-vari\'et\'e $D$ de $V$ correspond au lieu o\`u elles ne sont pas ind\'ependantes, et est donc d\'ecrit par l'annulation 
d'un certain nombre de mineurs. On en d\'eduit ais\'ement que ce lieu est de codimension $\geq 2$ dans $V$.
Par le corollaire \ref{cridef}, $X_A$ est alors d\'efective.
\end{proof}

On obtient alors une preuve du th\'eor\`eme \ref{princ}.

\begin{proof}[$\mathbf{Preuve \text{ }du \text{ }th\acute{e}or\grave{e}me\text{ }\ref{princ}}$]
On distingue deux cas.
\begin{enumerate}[(i)]
 \item 
Supposons qu'on n'ait pas $d_1=\ldots=d_c=1$ et $c<N+1$. Alors, par le lemme \ref{nondef}, $X_A$ n'est pas d\'efective, et 
la preuve du th\'eor\`eme \ref{princ} en caract\'eristique nulle fonctionne encore. Il faut seulement remplacer le 
th\'eor\`eme \ref{discri} par le th\'eor\`eme \ref{discribis} $(i)$, et \'evaluer $\mu$ \`a l'aide de la proposition \ref{mu}.
\item
Si $d_1=\ldots=d_c=1$ et $c<N+1$, $X_A$ est d\'efective par le lemme \ref{def}. Par le corollaire \ref{cridef}, 
$D$ est de codimension $\geq 2$ dans $V$, de sorte que $\Delta=1$, et que tous les degr\'es qu'on cherche \`a
calculer sont nuls. D'autre part, la preuve fournie en caract\'eristique nulle fonctionne sans modifications en caract\'eristique quelconque, 
en faisant intervenir le th\'eor\`eme \ref{discribis} $(ii)$ \`a la place du th\'eor\`eme \ref{discri}. Les
termes de droite dans l'\'enonc\'e du th\'eor\`eme \ref{princ} sont donc \'egalement nuls. Le facteur $\frac{1}{\mu}$ ne joue alors aucun r\^ole, et le th\'eor\`eme
est d\'emontr\'e.\end{enumerate}
\end{proof}

\subsection{R\'eduction modulo $p$ du discriminant}
On montre dans ce paragraphe le th\'eor\`eme \ref{princ2}.

\begin{proof}[$\mathbf{Preuve \text{ }du \text{ }th\acute{e}or\grave{e}me\text{ }\ref{princ2}}$]
La situation d\'ecrite dans le paragraphe \ref{situation} se met en famille sur $\Spec(\mathbb{Z})$ : on dispose d'un fibr\'e vectoriel g\'eom\'etrique 
$V_{\mathbb{Z}}$ sur $\Spec(\mathbb{Z})$, et d'un sous-sch\'ema ferm\'e r\'eduit $D_{\mathbb{Z}}$ de celui-ci. On note $\Delta_{\mathbb{Q}}$ et $\Delta_{\mathbb{F}_p}$
les polyn\^omes discriminant sur $\mathbb{Q}$ et $\mathbb{F}_p$.

Si $d_1=\ldots=d_c=1$ et $c<N+1$, $\Delta_{\mathbb{Q}}=1$ par le lemme \ref{def} et le corollaire \ref{cridef}, et le th\'eor\`eme est \'evident. 
Dans le cas contraire, toutes les fibres sont des hypersurfaces 
par le lemme \ref{nondef} et le corollaire \ref{cridef}. Elles co\"incident ensemblistement avec le lieu discriminant, et sont donc
irr\'eductibles par le lemme \ref{irred}. Par cons\'equent, la r\'eduction modulo $p$ de $\Delta_{\mathbb{Q}}$ s'annule pr\'ecis\'ement sur le lieu discriminant,
et est donc une puissance de $\Delta_{\mathbb{F}_p}$. Comparant les degr\'es \`a l'aide du th\'eor\`eme \ref{princ}, 
on voit que cette puissance est $1$ sauf si $p=2$ et $n$ est pair, auquel cas cette puissance vaut $2$. Comme $\Delta_{\mathbb{F}_p}$ est irr\'eductible par 
d\'efinition, cela conclut.
\end{proof}


\addcontentsline{toc}{section}{R\'{e}f\'{e}rences}

\bibliographystyle{plain}
\bibliography{biblio}

\end{document}